\documentclass[10pt]{article}
\usepackage{latexsym,amssymb,amsmath,amsthm,amsfonts}
\usepackage{graphicx}
\usepackage{hyperref}
\usepackage{color}

\newtheorem{lemma}{Lemma}[section]
\newtheorem{proposition}[lemma]{Proposition}
\newtheorem{theorem}[lemma]{Theorem}

\newtheorem{definition}[lemma]{Definition}
\newtheorem{example}[lemma]{Example}

\numberwithin{equation}{section}

\newcommand{\ud}{\mathrm{d}}
\newcommand{\RR}{\mathbb{R}}
\newcommand{\f}{\frac}

\newcommand{\xx}{|x|^2}

\newcommand{\xy}{\langle x,y\rangle}

\newcommand{\pp}[2]{\frac{\partial{#1}}{\partial{#2}}}
\newcommand{\ppp}[3]{\frac{\partial^2{#1}}{\partial{#2}\partial{#3}}}
\newcommand{\pppp}[4]%
  {\frac{\partial^3{#1}}{\partial{#2}\partial{#3}\partial{#4}}}

\newcommand{\p}{\phi}
\newcommand{\ps}{\phi(s)}
\newcommand{\pab}{\alpha\phi(\frac{\beta}{\alpha})}

\newcommand{\gab}{\alpha\phi(b^2,\frac{\beta}{\alpha})}
\newcommand{\pt}{\phi_2}
\newcommand{\po}{\phi_1}
\newcommand{\ptt}{\phi_{22}}
\newcommand{\pot}{\phi_{12}}

\renewcommand{\a}{\alpha}
\renewcommand{\b}{\beta}
\newcommand{\ab}{(\alpha,\beta)}
\newcommand{\ta}{\tilde\alpha}
\newcommand{\tb}{\tilde\beta}
\newcommand{\ha}{\hat\alpha}
\newcommand{\hb}{\hat\beta}
\newcommand{\ba}{\bar\alpha}
\newcommand{\bb}{\bar\beta}

\newcommand{\aij}{a_{ij}}

\newcommand{\bi}{b_i}
\newcommand{\bj}{b_j}

\newcommand{\bbj}{\bar b_j}

\newcommand{\bij}{b_{i|j}}
\newcommand{\taij}{\tilde a_{ij}}

\newcommand{\tbij}{\tilde b_{i|j}}
\newcommand{\haij}{\hat a_{ij}}

\newcommand{\hbij}{\hat b_{i|j}}
\newcommand{\baij}{\bar a_{ij}}

\newcommand{\bbij}{\bar b_{i|j}}

\newcommand{\G}{G^i_\alpha}
\newcommand{\tG}{\tilde G^i_{\tilde\alpha}}
\newcommand{\hG}{\hat G^i_{\hat\alpha}}
\newcommand{\bG}{\bar G^i_{\bar\alpha}}

\newcommand{\rij}{r_{ij}}
\newcommand{\sij}{s_{ij}}
\newcommand{\ri}{r_i}
\newcommand{\si}{s_i}
\newcommand{\rj}{r_j}
\newcommand{\sj}{s_j}

\begin{document}
\title{On dually flat general $\ab$-metrics
}
\author{Changtao Yu}
\date{}
\maketitle

\begin{abstract}
Based on the previous researches, in this paper we study the dual flatness of a special class of Finsler metrics called general $\ab$-metrics, which is defined by a Riemannian metric $\a$ and a $1$-form $\b$. A characterization for such metrics to be locally dually flat under some suitable conditions is provided. Many non-trivial explicit examples are constructed by a new kind of deformation technique. Moreover, the relationship of dual flatness and project flatness of such metrics are shown.
\end{abstract}

\section{Introduction}
In 2000, S.-I. Amari and H. Nagaoka introduced the notion of dual flatness for Riemannian metrics when they studied information geometry\cite{AN1}, and this notion have been extended to general Finsler metrics by Z. Shen in 2007\cite{szm-rfgw}.
A Finsler metric $F$ on a manifold $M$ is said to be {\it locally dually flat} if and only if in an {adapted} local coordinate system $(x^i)$, the function $F=F(x,y)$ satisfies
\begin{eqnarray}\label{dff}
[F^2]_{x^ky^l}y^k-2[F^2]_{x^l}=0.
\end{eqnarray}

For a Riemannian metric $\a=\sqrt{a_{ij}(x)y^iy^j}$, it is known that $\a$ is locally dually flat if and only if in an adapted coordinate system, its fundamental tensor is the Hessian of some local smooth function $\psi(x)$\cite{AN1}, i.e.,
\begin{eqnarray*}
a_{ij}(x)=\ppp{\psi}{x^i}{x^j}(x).
\end{eqnarray*}
In fact, the dual flatness of a Riemannian metric can also be described by its spray\cite{yct-odfr}: {\it $\a$ is locally dually flat if its spray coefficients can be expressed in some coordinate system as}
\begin{eqnarray}\label{duallyflat}
\G=2\theta y^i+\a^2\theta^i,
\end{eqnarray}
{\it where $\theta:=\theta_i(x)y^i$ is a $1$-form and  $\theta^i:=a^{ij}\theta_j$}.

The characterization for dually flat Riemannian metrics is clear. Hence, how to describe the dual flatness for Finsler metrics becomes an interesting problem. However, it is still not easy to be researched for the general case. So we begin with some special kinds of Finsler metrics.

Randers metrics, introduced by a physicist G. Randers in 1941 when he studied general relativity, is an important class of Finsler metrics. Generally, a Randers metric is given in the form $F=\a+\b$ where $\a$ is a Riemannian metric and $\b$ is a $1$-form. But it can also be expressed in another famous form as follows:
\begin{eqnarray}\label{NPE}
F=\f{\sqrt{(1-\bar b^2)\ba^2+\bb^2}}{1-\bar b^2}-\f{\bb}{1-\bar b^2},
\end{eqnarray}
where $\ba$ is also a Riemannian metric, $\bb$ is a $1$-form and $\bar b$ is the norm of $\bb$ with respect to $\ba$. $(\ba,\bb)$ is called the {\it navigation data} of the Randers metric $F$. Based on the results in \cite{csz-oldf}, the author provided a more clear description for dually flat Randers metrics: {\it A Randers metric $F=\a+\b$ is locally dually flat if and only if $\ba$ is locally dually flat and $\bb$ is dually related with respect to $\ba$}\cite{yct-odfr}. The notion of dually related $1$-forms was proposed by the author in \cite{yct-odfr}. The definition is given below:

\begin{definition}
Let $\a$ be a locally dually flat Riemannian metric on a manifold $M$. Suppose that the spray coefficients $\G$ of $\a$ are given in an adapted coordinate system by (\ref{duallyflat}) with some $1$-form $\theta$ on $M$. Then a $1$-form $\b$ on $M$ is said to be {dually related} with respect to $\a$ if
\begin{eqnarray*}
\bij=2\theta_i\bj+c(x)\aij,
\end{eqnarray*}
where $c(x)$ is a scalar function on $M$.
\end{definition}

As a generalization of Randers metrics from the
algebraic point of view, $\ab$-metrics are also defined by a Riemannian metric and a $1$-form and given in the form
$$F=\pab,$$
where $\ps$ is a smooth function. Because of its computability\cite{bcs}, many encouraging results about $\ab$-metrics have been achieved\cite{bcs,lb-szm-onac,szm-oaco,yct-dhfp}.

Recently, Q. Xia gave a local characterization of locally dually flat $\ab$-metrics with dimension $n\geq3$\cite{xoldf}. Later on, the author provide a more clear characterization. The result is much similar to that of Randers metrics: {\em If $F=\pab$ is a non-trivial locally dually flat $\ab$-metric with $n\geq3$, then after some special deformations, $\a$ will turn to be a locally dually flat Riemannian metric $\ba$ and $\b$ will turn to be a $1$-form $\bb$ which is dually related with respect to $\ba$. In this case, $F$ can be reexpressed as the form $F=\phi(\bar b^2,\frac{\bb}{\ba})$}\cite{yct-odfa}.

One can see that the navigation expression (\ref{NPE}) of Randers metrics is also given in the form
$$F=\gab.$$
Actually, such kind of Finsler metrics are belong to the metrical category called {\em general $\ab$-metric}, which is introduced by the author as a generalization of Randers metrics from the
geometric point of view\cite{yct-zhm-onan}. General $\ab$-metrics include not only all the $\ab$-metrics and the spherically symmetric Finsler metrics\cite{huang2} naturally, but also part of Bryant's metrics\cite{Br2,yct-zhm-onan} and fourth root metrics\cite{lb-szm-pffr}.

The main purpose of this paper is to describe and construct dually flat general $\ab$-metrics. It must be declare firstly that we will assume additionally that $\a$ is dually flat and $\b$ is dually related to $\a$. According to the discussions for Randers metrics and $\ab$-metrics, one can see that the dual flatness of a Randers metric or a $\ab$-metric is always arises from that of some Riemannian metric, and it is the dually related $1$-forms preserve the dual flatness. Hence, we believe that the assumption here is reasonable and appropriate.

The main result is given below:

\begin{theorem}\label{main1}
Let $F=\gab$ be a Finsler metric on an open domain $U\subseteq\RR^n$. Suppose that $\a$ and $\b$ satisfy
\begin{eqnarray}\label{conditions}
\G=2\theta y^i+\a^2\theta^i,\qquad\bij=c(x)\aij+2\theta_i\bj,
\end{eqnarray}
where $\theta$ is a $1$-form and $c(x)$ is a scalar function such that $c(x)\neq-2\theta_kb^k$. Then $F$ is dually flat on $U$ if and only if the function $\phi$ satisfies the following PDE:
\begin{eqnarray}\label{pde1}
\pt^2+\p\ptt+2s\po\pt+2s\p\pot-4\p\po=0.
\end{eqnarray}
\end{theorem}

Notice that $\phi_1$ means the derivation of $\phi$ with respect to the first variable $b^2$.

It should be pointed out that if the scale function $c(x)$ satisfies $c(x)=-2\theta_kb^k$, then according to the proof of Theorem \ref{main1}, $F=\gab$ will be always dually flat for any function $\phi(b^2,s)$. So it will be regarded as a trivial case. In another word, such a dually related $1$-form is trivial.

We have reason to conjecture that all the dually flat general $\ab$-metrics can only be obtained by this way. More specifically, we guess that if $F=\gab$ is a locally dually flat Finsler metric on a manifold with dimension $n\geq3$, then after necessary reexpressing in a new form $F=\ba\bar\phi(\bar b^2,\frac{\bb}{\ba})$, $\ba$ must be locally dually flat and $\bb$ must be dually related with respect to $\ba$. Moreover, the function $\bar\phi$ must satisfy (\ref{pde1}) if $\bb$ is non-trivial. Actually, it is true for all the locally dually flat $\ab$-metrics. One can check it by Maple combining with the discussions in \cite{yct-odfa}.

The solutions of (\ref{pde1}) are completely determined by Proposition \ref{pppp}. As a corollary, we obtain the following result, which reveals the closely relationship between the dual flatness and the projective flatness of a general $\ab$-metric. It is different from the related result in \cite{lb}.

\begin{theorem}\label{main2}
If $F=\gab$ is a dually flat general $\ab$-metrics on $U\subseteq\RR^n$ in which $\a$ and $\b$ satisfy (\ref{conditions}), then $\check F=\check\a(\phi^2)_2(\check b^2,\frac{\check\b}{\check\a})$ is projectively flat in which $\check\a$ is projectively flat and $\check\b$ is closed and conformal with respect to $\check\a$ so long as $\check F$ is a Finsler metric.

Conversely, if $\check F=\check\a\check\phi(\check b^2,\frac{\check\b}{\check\a})$ id a projectively flat general $\ab$-metrics on $U\subseteq\RR^n$ in which $\check\a$ is projective flat and $\check\b$ is closed and conformal with respect to $\check\a$, then $F=\gab$ is dually flat in which $\a$ and $\b$ satisfy (\ref{conditions}), $\phi$ is given~($C$ is a constant)
\begin{eqnarray}\label{ptod}
\phi(b^2,s)=\sqrt{\int_0^s\check\phi(b^2,\varsigma)\,\ud\varsigma+\f{1}{4}\int_0^{b^2}\check\phi_2(\iota,0)\,\ud\iota+C},
\end{eqnarray}
so long as $F$ is a Finsler metric.
\end{theorem}

In Section \ref{dd}, we will construct some dually flat Riemannian metrics and their dually related $1$-forms by the data
\begin{eqnarray}\label{aabb}
\a=|y|,\quad\b=\lambda\langle x,y\rangle+\langle a,y\rangle,
\end{eqnarray}
where $\lambda$ is a constant number and $a$ is a constant vector. The corresponding result is given by Proposition \ref{dddd}. One can get infinity many dually flat general $\ab$-metrics combining with Theorem \ref{main1}, Proposition \ref{dddd} and Proposition \ref{pppp}. Two typical examples are listed below, one can fine more examples in Section 5.

The Finsler metric
$$F=\f{\sqrt[4]{1+(\mu+\sigma^2)b^2}\sqrt{(1+\mu b^2)\a^2-\mu\b^2}}{1+\mu b^2}+\f{\sigma\b}{(1+\mu b^2)\sqrt[4]{1+(\mu+\sigma^2)b^2}}$$
are dually flat Randers metrics, where $\a$ and $\b$ are given by (\ref{aabb}), $\mu$ and $\sigma$ are constant numbers. When $\mu=-1$ and $\sigma=1$, the corresponding metrics are the famous generalized Funk metrics, whose dual flatness were first proved in \cite{csz-oldf}

The Finsler metric
\begin{eqnarray}\label{oneone}
F=\f{\left(\sqrt{1+(\mu+\sigma^2)b^2}\cdot\sqrt{(1+\mu b^2)\a^2-\mu\b^2}+\sigma\b\right)^\frac{3}{2}}{(1+\mu b^2)^\frac{3}{2}\sqrt[4]{(1+\mu b^2)\a^2-\mu\b^2}}
\end{eqnarray}
are dually flat general $\ab$-metrics, where $\a$ and $\b$ are given by (\ref{aabb}), $\mu$ and $\sigma$ are constant numbers. Such kind of general $\ab$-metrics are actually belong to the category of $\ab$-metrics because they can be reexpressed as
$$F=\sqrt{\f{(\breve\a+\breve\b)^3}{\breve\a}},$$
so long as we take
$$\breve\a=\f{\left(1+(\mu+\sigma^2)b^2\right)^\frac{3}{4}\sqrt{(1+\mu b^2)\a^2-\mu\b^2}}{(1+\mu b^2)^\frac{3}{2}},\qquad\breve\b=\f{\sigma\left(1+(\mu+\sigma^2)b^2\right)^\frac{1}{4}\b}{(1+\mu b^2)^\frac{3}{2}}.$$

By the way, it should be pointed out that when $\mu=-1$, $\lambda=\sigma=1$ and $a=0$, the above metric is given by
$$F=\sqrt{\frac{(\sqrt{(1-|x|^2)|y|^2+\langle x,y\rangle^2}+\langle x,y\rangle)^3}{(1-|x|^2)^3\sqrt{(1-|x|^2)|y|^2+\langle x,y\rangle^2}}}.$$
It is also obtained by B. Li independently in another different way\cite{lb}.

Finally, when $a=0$, all the dually flat Finsler metrics obtained by this way are just the so-called symmetric Finsler metrics, which are given in the form $F=\phi(|x|^2,\frac{\langle x,y\rangle}{|y})$\cite{huang2}.

\section{Preliminaries}
Let $F$ be a Finsler metric on a $n$-dimensional manifold $M$. The geodesic spray coefficients of $F$ are defined by
$$G^i:=\f{1}{4}g^{il}\left\{[F^2]_{x^ky^l}y^k-[F^2]_{x^l}\right\},$$
where $g^{ij}$ is the inverse of the fundamental tensor $g_{ij}:=[\frac{1}{2}F^2]_{y^iy^j}$. For a Riemannian metric, the spray coefficients are determined by its Christoffel symbols as $G^i(x,y)=\frac{1}{2}\Gamma^i{}_{jk}(x)y^jy^k$.

Suppose that $\phi(b^2,s)$ is a positive smooth function  defined on the domain $|s|\leq b<b_o$ for some positive number (maybe infinity) $b_o$. Then the function $F=\gab$ determines a Finsler metric for any Riemannian metric $\a=\sqrt{a_{ij}(x)y^iy^j}$ and any $1$-form $\b=b_i(x)y^i$ if and only if $\p(b^2,s)$ satisfies
\begin{eqnarray*}
\p-s\pt>0,\quad\p-s\pt+(b^2-s^2)\ptt>0
\end{eqnarray*}
when $n\geq3$ or
\begin{eqnarray*}
\p-s\pt+(b^2-s^2)\ptt>0
\end{eqnarray*}
when $n=2$\cite{yct-zhm-onan}. Such kind of Finsler metrics belong to the metrical category called {\em general $\ab$ metrics}. For a given metric, $b:=\|\b\|_\a$ is the norm of $\b$.

In Section \ref{dd}, we need a special kind of metric deformations called {\em $\b$-deformations}\cite{yct-dhfp,yct-odfr}, which are determined by a Riemannian metric $\a$ and a $1$-form $\b$ and listed below:
\begin{eqnarray*}
&\ta=\sqrt{\a^2-\kappa(b^2)\b^2},\qquad\tb=\b;\\
&\ha=e^{\rho(b^2)}\ta,\qquad\hb=\tb;\\
&\ba=\ha,\qquad\bb=\nu(b^2)\hb.
\end{eqnarray*}
In order to keep the positive definition of $\ta$, the deformation factor $\kappa(b^2)$ must satisfies an additional condition:
\begin{eqnarray}\label{conditiononk}
1-\kappa b^2>0.
\end{eqnarray}

Some basic formulas of $\b$-deformations are listed below. It should be attention that the notation `$\dot b_{i|j}$' always means the covariant derivative of the $1$-form `$\dot\b$' with respect to the corresponding Riemannian metric `$\dot\a$', where the symbol `~$\dot{}$~' can be `~$\tilde{}$~', `~$\hat{}$~', `~$\bar{}$~' or empty in this paper. Moreover, we need the following abbreviations,
\begin{eqnarray*}
&r_{00}:=r_{ij}y^iy^j,~r_i:=r_{li}b^l,~r^i:=a^{ij}r_j,~r_0:=r_iy^i,~r:=r_ib^i,\\
&s_{i0}:=s_{ij}y^j,~s^i{}_0:=a^{ij}s_{j0},~s_i:=s_{li}b^l,~s^i:=a^{ij}s_j,~s_0:=s_ib^i,
\end{eqnarray*}
where $\rij$ and $\sij$ are given by $\rij:=\frac{1}{2}(\bij+b_{j|i})$ and $\sij:=\frac{1}{2}(\bij-b_{j|i})$.

\begin{lemma}\cite{yct-dhfp}\label{beta1}
Let $\ta=\sqrt{\a^2-\kappa(b^2)\b^2}$, $\tb=\b$. Then
\begin{eqnarray*}
\tG&=&\G-\frac{\kappa}{2(1-\kappa b^2)}\big\{2(1-\kappa b^2)\b s^i{}_0+r_{00}b^i+2\kappa s_0\b b^i\big\}\\
&&+\frac{\kappa'}{2(1-\kappa b^2)}\big\{(1-\kappa b^2)\b^2(r^i+s^i)+\kappa r\b^2b^i-2(r_0+s_0)\b
b^i\big\},\\
\tbij&=&\bij+\frac{\kappa}{1-\kappa b^2}\big\{b^2\rij+\bi\sj+\bj\si\big\}
-\frac{\kappa'}{1-\kappa b^2}\big\{r\bi\bj-b^2\bi(\rj+\sj)-b^2\bj(\ri+\si)\big\}.
\end{eqnarray*}
\end{lemma}

\begin{lemma}\cite{yct-dhfp}\label{beta2}
Let $\ha=e^{\rho(b^2)}\ta$, $\hb=\tb$. Then
\begin{eqnarray*}
\hG&=&\tG+\rho'\left\{2(r_0+s_0)y^i-(\a^2-\kappa \b^2)\left(r^i+s^i+\frac{\kappa}{1-\kappa b^2}rb^i\right)\right\},\\
\hbij&=&\tbij-2\rho'\left\{\bi(\rj+\sj)+\bj(\ri+\si)-\frac{1}{1-\kappa b^2}r(\aij-\kappa\bi\bj)\right\}.
\end{eqnarray*}
\end{lemma}

\begin{lemma}\cite{yct-dhfp}\label{beta3}
Let $\ba=\ha$, $\bb=\nu(b^2)\hb$. Then
\begin{eqnarray*}
\bG=\hG,\qquad\bbij=\nu\hbij+2\nu'\bi(\rj+\sj).
\end{eqnarray*}
\end{lemma}

\section{Proof of Theorem \ref{main1}}
Suppose $\a$ and $\b$ satisfy (\ref{conditions}). It is easy to verify that
\begin{eqnarray*}
\rij&=&c\aij+\theta_i\bj+\theta_j\bi,\\
\sij&=&\theta_i\bj-\theta_j\bi,\\
(b^2)_{x^l}&=&2(r_l+s_l)=2\bar cb_l,\\
\a_{x^l}&=&\f{y_m}{\a}\pp{G^m_\a}{y^l}=2\a^{-1}(\a^2\theta_l+2\theta y_l),\\
\b_{x^l}&=&b_{m|l}y^m+b_m\pp{G^m_\a}{y^l}=\bar cy_l+2\b\theta_l+4\theta b_l,\\
s_{y^l}&=&\a^{-2}(\a b_l-sy_l),
\end{eqnarray*}
where $\bar c=(c+2\theta_kb^k)$ and $y_i=\aij y^j$. Combining with the above equalities we obtain
\begin{eqnarray}
[F^2]_{x^l}&=&2\p^2\a\a_{x^l}+2\p\po\a^2\cdot2(r_l+s_l)+2\p\pt\a^2\cdot\f{\a\b_{x^l}-\b\a_{x^l}}{\a^2}\nonumber\\
&=&2\p\left\{2(\p-s\pt)(\a^2\theta_l+2\theta y_l)+2\po\bar c\a^2 b_l+\pt\a(\bar c y_l+2\b\theta_l+4\theta b_l)\right\}\label{temp1}
\end{eqnarray}
and
\begin{eqnarray*}
[F^2]_{x^ky^l}&=&2\pt\left\{2(\p-s\pt)(\a^2\theta_k+2\theta y_k)+2\po\bar c\a^2 b_k+\pt\a(\bar c y_k+2\b\theta_k+4\theta b_k)\right\}\f{\a b_l-sy_l}{\a^2}\nonumber\\
&&+2\p\left\{-2s\ptt(\a^2\theta_k+2\theta y_k)+2\pot\bar c\a^2 b_k+\ptt\a(\bar c y_k+2\b\theta_k+4\theta b_k)\right\}\f{\a b_l-sy_l}{\a^2}\nonumber\\
&&+2\p\big\{4(\p-s\pt)(\theta_ky_l+y_k\theta_l+\theta a_{kl})+4\po\bar cb_ky_l+\pt(\bar c\a^{-1}y_ky_l+\bar c\a a_{kl}+2s\theta_ky_l\nonumber\\
&&+2\a\theta_kb_l+4\a^{-1}\theta b_ky_l+4\a b_k\theta_l)\big\}.
\end{eqnarray*}
Hence,
\begin{eqnarray}
[F^2]_{x^ky^l}y^k&=&2\pt\left\{6\p\theta+2\po\bar c\b+\pt\bar c\a\right\}(\a b_l-sy_l)+2\p\left\{2\pot\bar c\b+\ptt\bar c\a\right\}(\a b_l-sy_l)\nonumber\\
&&+4\p\left\{2(\p-s\pt)(\a^2\theta_l+2\theta y_l)+2\po\bar c\b y_l+\pt(\bar c\a y_l+3s\theta y_l+\a\theta b_l+2\a\b\theta_l\right\}\label{temp2}.
\end{eqnarray}

By (\ref{temp1}) and (\ref{temp2}), one can see that (\ref{dff}) is equivalent to
$$\bar c(\pt^2+\p\ptt+2s\po\pt+2s\p\pot-4\p\po)(\a b_l-sy_l)=0,$$
hence (\ref{pde1}) holds if $\bar c\neq0$.

\section{Dually flat Riemannian metrics and dually related $1$-forms}\label{dd}
Let $\a$ be a Riemannian metric with constant sectional curvature $\mu$, and $\b$ be a closed $1$-form which is also conformal with respect to $\a$, then there is a local coordinate system in which $\a$ and $\b$ can be expressed as\cite{yct-dhfp}
\begin{eqnarray}
&\displaystyle\a=\frac{\sqrt{(1+\mu|x|^2)|y|^2-\mu\langle x,y\rangle^2}}{1+\mu|x|^2},&\label{csc}\\
&\displaystyle\b=\frac{\lambda\langle x,y\rangle+(1+\mu|x|^2)\langle a,y\rangle-\mu\langle a,x\rangle\langle x,y\rangle}{(1+\mu|x|^2)^\frac{3}{2}},\label{cc}&
\end{eqnarray}
where $\lambda$ is a constant number and $a$ is a constant vector. Moreover,
\begin{eqnarray*}
\bij=\f{\lambda-\mu\langle a,x\rangle}{\sqrt{1+\mu|x|^2}}\aij.
\end{eqnarray*}
These special Riemannian metrics and $1$-forms play an important role in projective Finsler geometry\cite{yct-dhfp,yct-zhm-onan}.

In the rest of this section, we will use the above data to construct some dually flat Riemannian metrics and their dually related $1$-forms by $\b$-deformations. Firstly, it seems impossible to obtain dually flat Riemannian metrics by data (\ref{csc}) and (\ref{cc}) without any additional condition. The reason is given below.

\begin{lemma}
Given $\a$ and $\b$ as (\ref{csc}) and (\ref{cc}). Then $\a$ cann't turn to be a dually flat Riemannian metric satisfying (\ref{duallyflat}) by $\b$-deformations unless $\mu=0$ or $a=0$.
\end{lemma}
\begin{proof}
It is known that $\G=Py^i$, where
\begin{eqnarray}\label{P}
P=-\frac{\mu\xy}{1+\mu\xx}.
\end{eqnarray}
Set $\tau(x)=\frac{\lambda-\mu\langle a,x\rangle}{\sqrt{1+\mu|x|^2}}$, then $\bij=\tau(x)\aij$.

Carry out the first step of $\b$-deformations. By Lemma \ref{beta1} we have
\begin{eqnarray*}
\tG=
Py^i-\f{\tau}{2(1-\kappa b^2)}(\kappa\a^2+\kappa'\b^2)b^i.
\end{eqnarray*}
Combining with Lemma \ref{beta2} one can see that $\ta$ cann't turn to be a dually flat metric $\ha$ satisfying (\ref{duallyflat}) by the second step of $\b$-deformations unless $\tG$ is in the form $\tG=Py^i+Q\ta^2 b^i$, which means that $\kappa$ must satisfy the following equation
\begin{eqnarray}\label{kappa}
\kappa'=-\kappa^2.
\end{eqnarray}
So we assume the deformation factor $\kappa$ satisfies (\ref{kappa}) from now on.

Carry out the second step of $\b$-deformations. By Lemma \ref{beta2} we have
\begin{eqnarray*}
\hG=(P+2\tau\rho'\b)y^i-\f{\tau}{2(1-\kappa b^2)}(\kappa+2\rho')e^{-2\rho}\ha^2 b^i.\label{hGtheta}
\end{eqnarray*}
Hence, $\ha$ is dually flat and satisfies (\ref{duallyflat}) if and only if
$$P+2\tau\rho'\b=-\f{\tau}{1-\kappa b^2}(\kappa+2\rho')e^{-2\rho}b^i\hat y_i,$$
where $\hat y_i:=\haij y^j$. The above equality is equivalent to
\begin{eqnarray}\label{prho}
P=-\tau(\kappa+4\rho')\b.
\end{eqnarray}
Combining with (\ref{cc}) and (\ref{P}), it is not hard to find that the equality (\ref{prho}) will not hold unless $\mu=0$~(in this case $P=0$) or $a=0$~(in this case $\b=\frac{\lambda\xy}{(1+\mu|x|^2)^\frac{3}{2}}$ is parallel to $P$).
\end{proof}

When $a=0$, it have been proven by $\b$-deformations in \cite{yct-odfr} that the Riemannian metrics
\begin{eqnarray}\label{dfR}
\ba=\frac{\sqrt{(1+\mu|x|^2)|y|^2-\mu\langle x,y\rangle^2}}{(1+\mu|x|^2)^\frac{3}{4}}
\end{eqnarray}
are dually flat on the ball $\mathbb B^n(r_\mu)$ and the $1$-forms
\begin{eqnarray}\label{drb}
\bb=\f{\lambda\xy}{(1+\mu|x|^2)^\frac{5}{4}}
\end{eqnarray}
are dually related to $\ba$ for any constant numbers $\mu$ and $\lambda$, where the the radius $r_\mu$ is given by $r_\mu=\frac{1}{\sqrt{-\mu}}$ if
$\mu<0$ and $r_\mu=+\infty$ if $\mu\geq0$.
The above result was obtain by taking the deformation factor $\kappa=0$.

Next, let's discuss another case. Taking $\mu=0$ in (\ref{csc}) and (\ref{cc}), we get the standard Euclidean metric and its closed conformal $1$-form,
\begin{eqnarray}\label{cccc}
\a=|y|,\qquad\b=\lambda\xy+\langle a,y\rangle.
\end{eqnarray}
In this case, $\G=0$ and $\bij=\lambda\aij$. It is obviously that $\a$ is dually flat and $\b$ is dually related to $\a$.

Assume the deformation factor $\kappa$ satisfies (\ref{kappa}), by Lemma \ref{beta1} again we get
$$\tG=-\f{\lambda\kappa}{2(1-\kappa b^2)}\ta^2b^i,\qquad\tbij=\f{\lambda}{1-\kappa b^2}\taij+\lambda\kappa\bi\bj.$$
By (\ref{prho}), $\ta$ can turn to be a dually flat metric $\ha$ satisfying (\ref{duallyflat}) by the second kind of $\b$-deformations if and only if
\begin{eqnarray*}
\kappa+4\rho'=0.
\end{eqnarray*}
So we assume the deformation factor $\rho$ satisfies the above equation. In this case, by Lemma \ref{beta2} again we get
$$\hG=2\hat\theta y^i+\ha^2\hat\theta^i,\qquad\hbij=\f{\lambda}{1-\kappa b^2}(1-\f{1}{2}\kappa b^2)e^{-2\rho}\haij+2\lambda\kappa\bi\bj,$$
where $\hat\theta=-\frac{1}{4}\lambda\kappa\beta$ and $\hat\theta^i:=\hat a^{ij}\hat\theta_j$.

Carry out the third step of $\b$-deformations, then
\begin{eqnarray*}
\bG=2\bar\theta y^i+\ba^2\bar\theta^i,\qquad\bbij=\f{\lambda}{1-\kappa b^2}\nu(1-\f{1}{2}\kappa b^2)e^{-2\rho}\baij+2\lambda(\kappa+\f{\nu'}{\nu})\bi\bbj,
\end{eqnarray*}
where $\bar\theta:=\hat\theta$ and $\bar\theta^i:=\bar a^{ij}\bar\theta_j$. Hence, $\bb$ is dually related to $\ba$ if and only if
$$\lambda(\kappa+\f{\nu'}{\nu})\bi=\bar\theta_i,$$
which means that the deformation factor $\nu$ must satisfy
\begin{eqnarray*}
\f{\nu'}{\nu}=-\f{5}{4}\kappa.
\end{eqnarray*}

Obviously, $\kappa=0$ is a solution of (\ref{kappa}). As a result, $\rho$ and $\nu$ are both constants. In this case, the corresponding $\b$-deformations are just a scaling for $\a$ and $\b$. It is trivial and should not be discussed.

When $\kappa\neq0$, the solutions of (\ref{kappa}) are given by
$$\kappa=\frac{1}{b^2+\textrm{constant}}.$$
In consideration of the additional condition of $\kappa$ (\ref{conditiononk}), $\kappa$ can only be chosen as $\kappa=\frac{1}{c+b^2}$ or $\kappa=-\frac{1}{c-b^2}$ where $c$ is a positive number. Hence, the Riemannian metric $(c\pm b^2)^{-\frac{1}{4}}\sqrt{\a^2\mp(c\pm b^2)^{-1}\b^2}$ is dually flat and the $1$-form $(c\pm b^2)^{-\frac{5}{4}}\b$ is dually related. It can be verified easily that such a $1$-form is a trivial dually related $1$-form if and only if $\lambda=0$.

After a scaling, we have proven the following result.
\begin{proposition}\label{dddd}
Let $\a$ and $\b$ are given by (\ref{cccc}). Then the Riemannian metric
$$\ba=\f{\sqrt{(1+\mu b^2)\a^2-\mu\b^2}}{(1+\mu b^2)^\frac{3}{4}}$$
is dually flat and the $1$-form
$$\bb=\f{\sigma\b}{(1+\mu b^2)^\frac{5}{4}}$$
is dually related to $\ba$, where $\mu$ and $\sigma$ are constant numbers.
\end{proposition}

It is obvious that the above result cover (\ref{dfR}) and (\ref{drb}).

\section{Solutions of Equation (\ref{pde1}) and Examples}
\begin{proposition}\label{pppp}
The non-zero solutions of the basic equation (\ref{pde1}) are given by
$$\phi(b^2,s)=\sqrt{g(b^2)+2g'(b^2)s^2+sf(b^2-s^2)+\int_0^s(s^2+\sigma^2)f'(b^2-\sigma^2)\,\ud\sigma},$$
where $f(t)$ and $g(t)>0$ are two arbitrary smooth functions.
\end{proposition}
\begin{proof}
Let $\psi:=\phi^2$, then $\psi$ satisfies
\begin{eqnarray}\label{psi22}
\psi_{22}+2s\psi_{12}-4\psi_1=0.
\end{eqnarray}
Differentiating the above equation with respect to $s$, one obtain
$$\psi_{222}+2s\psi_{122}-2\psi_{12}=0.$$
Let $\varphi:=\psi_2$, then $\varphi$ satisfies
\begin{eqnarray}\label{varphi}
\varphi_{22}+2s\varphi_{12}-2\varphi_1=0.
\end{eqnarray}
Considering the variable substitution
$$u=b^2-s^2,\quad v=s,$$
then $b^2=u+v^2$, $s=v$. So by (\ref{varphi}) we have
$$\f{\partial}{\partial v}(\varphi-s\varphi_2)=0,$$
which means $\varphi-s\varphi_2=f(b^2-s^2)$ is just a function of $b^2-s^2$. Hence, $\psi_2$ is given by
\begin{eqnarray*}
\psi_2=h(b^2)s+f(b^2-s^2)+2s\int_0^sf'(b^2-\sigma^2)\,\ud\sigma
\end{eqnarray*}
for some function $h(b^2)$. As a result, $\psi$ is given by
\begin{eqnarray}\label{psi}
\psi=g(b^2)+\f{1}{2}h(b^2)s^2+sf(b^2-s^2)+\int_0^s(s^2+\sigma^2)f'(b^2-\sigma^2)\,\ud\sigma
\end{eqnarray}
for some function $g(b^2)$. Putting (\ref{psi}) into (\ref{psi22}) yields $h=4g'$.

Finally, by taking $s=0$ we know that the function $g$ must be positive.
\end{proof}

Some typical examples are listed below.

\begin{example}
Take $f(t)=0$, then
$$\phi(b^2,s)=\sqrt{g(b^2)+2g'(b^2)s^2}.$$
In this case, we obtain a class of dually flat Riemannian metrics
$$F=\sqrt{g(b^2)\a^2+2g'(b^2)\b^2},$$
where $\a$ and $\b$ satisfy (\ref{conditions}).
\end{example}

\begin{example}
Take $f(t)=\f{2\varepsilon}{(1-\kappa t)^\frac{3}{2}}$ and $g(t)=\f{1}{1-\kappa t}$, then
$$\phi(b^2,s)=\f{\sqrt{1-\kappa b^2+2\kappa s^2+2\varepsilon s\sqrt{1-\kappa b^2+\kappa s^2}}}{1-\kappa b^2}.$$
In particular, $\phi(b^2,s)=\frac{\sqrt{1-b^2+s^2}+s}{1-b^2}$ when $\kappa=1$ and $\varepsilon=1$, in this case, the corresponding general $\ab$-metrics are Randers metrics. $\phi(b^2,s)=\sqrt{1+s}$ when $\kappa=0$ and $\varepsilon=\frac{1}{2}$, in this case, the corresponding general $\ab$-metrics are those $\ab$-metrics given in the form $F=\sqrt{\a(\a+\b)}$.
\end{example}

\begin{example}
Take $f(t)=3(1+t)$ and $g(t)=(1+t)^\frac{3}{2}$, then
$$\phi(b^2,s)=(\sqrt{1+b^2}+s)^\frac{3}{2}.$$
In this case, the corresponding dually flat Finsler metrics determined by Theorem \ref{main1} and Proposition \ref{dddd} are given by (\ref{oneone}).
\end{example}

\begin{example}
Take $f(t)=\f{2}{\sqrt{1+t}}$ and $g(t)=\ln\sqrt{1+t}$, then
$$\phi(b^2,s)=\sqrt{\ln\sqrt{1+b^2}+\arcsin\f{s}{\sqrt{1+b^2}}+\f{s(\sqrt{1+b^2-s^2}+s)}{1+b^2}}.$$
\end{example}

\begin{example}
Take $f(t)=\f{2}{\sqrt{1-t}}$ and $g(t)=\ln\sqrt{1-t}$, then
$$\phi(b^2,s)=\sqrt{\mathrm{arcsinh}\,\f{s}{\sqrt{1-b^2}}+\f{s}{\sqrt{1-b^2+s^2}+s}}.$$
\end{example}

\begin{example}
Take $f(t)=4\sqrt{1+t}$ and $g(t)=1$, then
$$\phi(b^2,s)=\sqrt{1+3s\sqrt{1+b^2-s^2}+(1+b^2+2s^2)\arcsin\f{s}{\sqrt{1+b^2}}}.$$
\end{example}

\begin{example}
Take $f(t)=4\sqrt{1-t}$ and $g(t)=1+\frac{1}{2}(1-t)\ln(1-t)$, then
$$\phi(b^2,s)=\sqrt{1+s^2+3s\sqrt{1-b^2+s^2}+(1-b^2-2s^2)\mathrm{arcsinh}\,\f{s}{\sqrt{1-b^2}}}.$$
\end{example}

\begin{example}
Take $f(t)=\Re\frac{2i}{\sqrt{e^{ip}+t}}$ and $g(t)=\Re\ln\sqrt{e^{ip}+t}$, then
$$\phi(b^2,s)=\sqrt{\Re\left\{\ln\sqrt{e^{ip}+b^2}+i\arctan\f{s}{\sqrt{e^{ip}+b^2-s^2}}
+\f{is}{\sqrt{e^{ip}+b^2-s^2}+is}\right\}}.$$
\end{example}

\section{Proof of Theorem \ref{main2}}
\begin{proof}[Proof of Theorem \ref{main2}]
By Theorem \ref{main1} we know that $\phi$ satisfies the basic equation (\ref{pde1}) and hence $(\phi)^2_2$ satisfies (\ref{varphi}). In other hand, we have proved in \cite{szm-yct-oacfe} that if $\check\a$ is projectively flat and $\check\b$ is closed and conformal with respect to $\check\a$, then the general $\ab$-metric $\check F=\check\a\check\phi(\check b^2,\frac{\check\b}{\check\a})$ is projectively flat if and only if the function $\check\phi$ satisfies
\begin{eqnarray}\label{checkp}
\check\phi_{22}=2(\check\phi_1-s\check\phi_{12}),
\end{eqnarray}
i.e. the equation (\ref{varphi}). So the first part of Theorem \ref{main2} holds immediately.

Conversely, by assumption and the fact mentioned above, we know that the function $\check\phi$ satisfies (\ref{checkp}). According to the argument in the proof of Proposition \ref{pppp}, the solutions of (\ref{checkp}) are given by
\begin{eqnarray*}
\check\phi(b^2,s)=h(b^2)s+f(b^2-s^2)+2s\int_0^sf'(b^2-\sigma^2)\,\ud\sigma.
\end{eqnarray*}
It is obvious that $\check\phi_2(b^2,0)=h(b^2)$, hence
\begin{eqnarray*}
\p(b^2,s)&=&\sqrt{\int_0^s\check\phi(b^2,\varsigma)\,\ud\varsigma+\f{1}{4}\int_0^{b^2}\check\phi_2(\iota,0)\,\ud\iota+C}\\
&=&\sqrt{\f{1}{4}\int_0^{b^2}h(\iota)\,\ud\iota+C+\f{1}{2}h(b^2)s^2+sf(b^2-s^2)+\int_0^s(s^2+\sigma^2)f'(b^2-\sigma^2)\,\ud\sigma}
\end{eqnarray*}
are just the solutions of (\ref{pde1}) by Proposition \ref{pppp}. So the second part of Theorem \ref{main2} holds.
\end{proof}

Finally, we provide an example of Theorem \ref{main2}. Taking $\phi(b^2,s)=(\sqrt{1+b^2}+s)^\frac{3}{2}$, then
$$\check\phi:=(\phi^2)_2=3(\sqrt{1+b^2}+s)^2,$$
In this case, the corresponding general $\ab$-metric $F=\check\a\check\phi(\check b^2,\frac{\check\b}{\check\a})$ constructed by Theorem \ref{main2} is just a Berwald type of metric after a scaling\cite{szm-yct-oesm}.

\noindent Changtao Yu\\
School of Mathematical Sciences, South China Normal
University, Guangzhou, 510631, P.R. China\\
aizhenli@gmail.com
\end{document}